\documentclass[12pt,twoside]{amsart}
\usepackage{amssymb}
\usepackage{amscd}
\usepackage{color}

\usepackage[normalem]{ulem} 
\newcommand\redout{\bgroup\markoverwith
{\textcolor{red}{\rule[.5ex]{2pt}{0.4pt}}}\ULon}

\title[Infinitely generated symbolic Rees algebras]
{Infinitely generated symbolic Rees algebras over finite fields} 
\author{Akiyoshi Sannai and Hiromu Tanaka} 
\address{Center for Advanced Intelligence Project RIKEN 1-4-1, Nihonbashi, Chuo, Tokyo 103-0027, Japan.}
\email{akiyoshi.sannai@riken.jp}
\address{Graduate School of Mathematical Sciences, 
the University of Tokyo, 3-8-1 Komaba, Meguro-ku, Tokyo 153-8914, Japan.}
\email{tanaka@ms.u-tokyo.ac.jp}

\newcommand{\sat}[0]{{\operatorname{sat}}}

\newcommand{\Proj}[0]{{\operatorname{Proj}}}
\newcommand{\Spec}[0]{{\operatorname{Spec}}}

\newcommand{\Supp}[0]{{\operatorname{Supp}}}

\newcommand{\Ex}[0]{{\operatorname{Ex}}}
\newtheorem{thm}{Theorem}[section]
\newtheorem{lem}[thm]{Lemma}

\newtheorem{prop}[thm]{Proposition}
\newtheorem{claim}[thm]{Claim}

\theoremstyle{definition}
\newtheorem{ques}[thm]{Question}
\newtheorem{conje}[thm]{Conjecture}
\newtheorem{dfn}[thm]{Definition}

\newtheorem{rem}[thm]{Remark}
\newtheorem*{ack}{Acknowledgements}      
\newtheorem{nota}[thm]{Notation}

\makeatletter
  
  \@addtoreset{equation}{section}
  \makeatother

\newcommand{\p}{\mathfrak{p}}
\newcommand{\q}{\mathfrak{q}}
\newcommand{\m}{\mathfrak{m}}

\newcommand{\Q}{\mathbb{Q}}
\newcommand{\Z}{\mathbb{Z}}
\newcommand{\F}{\mathbb{F}}
\newcommand{\MO}{\mathcal{O}}
\begin{document}

\maketitle

\begin{abstract}
For the polynomial ring over an arbitrary field with twelve variables, 
there exists a prime ideal whose symbolic Rees algebra is not finitely generated. 
\end{abstract}

\tableofcontents

\section{Introduction}
Let $A$ be a polynomial ring over a field $k$ with finitely many variables. 
For a field $L$ satisfying $k\subset L \subset {\rm Frac}\,(A)$, 
Hilbert's fourteenth problem asks whether or not the ring $L\cap A$ is finitely generated over $k$. 
In 1958, Nagata found the first counterexample to this problem 
over arbitrary sufficiently large fields \cite{Nag58}. 
For more examples, we refer to \cite{Rob90'}, \cite{Kur05} and \cite{Tot08}. 
On the other hand, this problem is related to the following question raised by Cowsik \cite{Cow85}.

\begin{ques}\label{q-Cowsik}
Let $A$ be a polynomial ring over a field with finitely many variables and let $P$ be a prime ideal of $A$. 
Set $P^{(m)}:=P^mA_P \cap A$. 
Then is the symbolic Rees algebra $R_S(P):=\bigoplus_{m=0}^{\infty} P^{(m)}$ a finitely generated $k$-algebra? 
\end{ques}

Indeed, Roberts settled Question \ref{q-Cowsik} negatively \cite{Rob85}, 
using Nagata's counterexample mentioned above. 
Roberts's construction is valid only over sufficiently large fields of characteristic zero, 
although Nagata's example is independent of the characteristic of the base field. 
This is because Roberts's proof requires a theorem of Bertini type that fails in positive characteristic (cf. \cite[line 7 in page 591]{Rob85}). 
On the other hand, it is known for experts that Roberts's method works, after suitable modifications, 
also for the case where $k$ is not algebraic over a finite field. 
Roughly speaking, counterexamples over such fields can be found after 
replacing the theorem of Bertini type and Nagata's counterexample used in \cite{Rob85} 
by \cite[Theorem 2.1]{DH91} and the blowup of $\mathbb P^2$ along general nine points, respectively. 
In this sense, Question \ref{q-Cowsik} is still open if $k$ is algebraic over a finite field.

The purpose of this paper is to give the negative answer to Question \ref{q-Cowsik} over an arbitrary base field. 
More specifically, the main theorem is as follows.

\begin{thm}[cf. Theorem \ref{t-main2}]\label{intro-t-main}
Let $k$ be a field. 
Let $A$ be the polynomial ring over $k$ with twelve variables. 
Then there exists a prime ideal $\p$ of $A$ 
whose symbolic Rees algebra $\bigoplus_{m=0}^{\infty} \p^{(m)}$ is not a noetherian ring. 
\end{thm}

\subsection{Sketch of the proof}

We overview some of the ideas used in the proof of Theorem \ref{intro-t-main}. 
Let us treat the case where $k=\F_p$. 
Our method is based on a geometric description of symbolic Rees algebras  
that was pointed out by Cutkosky in a certain special case \cite{Cut91}. 
We start with a projective smooth surface $V$ over $\F_p$, constructed by Totaro, 
that has a nef divisor $M$ which is not semi-ample. 
We embed $V$ into the eleven-dimensional projective space $\mathbb P^{11}_{\F_p}$ 
(cf. Lemma \ref{l-totaro}). 
Thanks to a theorem of Bertini type over finite fields, 
we can find a smooth curve $W$ on $V$ that is linearly equivalent to $stH|_V-tM$ 
for a hyperplane divisor $H$ of $\mathbb P^{11}_{\F_p}$ under the assumption that $t\gg s\gg 0$. 
Take a homogeneous prime ideal $\p$ on $A=\F_p[x_0, \cdots, x_{11}]$ that defines $W$. 
Let $f:X \to \mathbb P^{11}_{\F_p}$ be the blowup along $W$. 
Set $D:=f^*H$ and let $E$ be the $f$-exceptional prime divisor on $X$. 
Then $\bigoplus_{m=0}^{\infty} \p^{(m)}$ is not a noetherian ring if and only if 
the Cox ring of $X$ is not a noetherian ring (cf. Proposition \ref{p-double-gr}). 
In particular it suffices to find a nef divisor on $X$ that is not semi-ample. 
By choosing $s$ and $t$ carefully, we can find such a divisor (cf. Proposition \ref{p-not-sa}(3)).
For more details, see Section \ref{s-main-thm}.

\subsection{Related topics}

It is worth mentioning that, concerning Question \ref{q-Cowsik},  
many authors have studied the case where $P$ is the prime ideal of $k[x, y, z]$ 
that 
defines a space monomial curve $(t^a, t^b, t^c)$ in $\mathbb A^3_k$. 
For instance, Goto, Nishida and Watanabe proved that for some triples $(a, b, c)$, 
the associated symbolic Rees algebras are not finitely generated if $k$ is of characteristic zero \cite{GNW94}. 
It is remarkable that this result is applied to study the compactified moduli space 
 $\overline{\mathcal M}_{0, n}$ of pointed rational curves. 
More specifically, it turns out that 
$\overline{\mathcal M}_{0, n}$ is not a Mori dream space if $n \geq 13$ 
and the base field is of characteristic zero \cite{Cas09}, \cite{GK16}. 

Since the case of characteristic zero has such an application, 
it is natural to consider also the case of positive characteristic.
However the situation seems to be subtler. 
Indeed, 
if the base field is of positive characteristic, 
then it is known that the analogous rings of 
the examples given in \cite{GNW94} and \cite{Rob90'} 
are shown to be finitely generated by \cite{Cut91} \cite{GNW94} and \cite{Kur93} \cite{Kur94}, respectively. 
Then Goto and Watanabe made the following conjecture, which still remains to be an open problem. 

\begin{conje}\label{GW}
Let $R$ be the polynomial ring over a field $k$ with three valuables. 
Let $P$ be the prime ideal that defines 
a space monomial curve $(t^a, t^b, t^c)$ in $\mathbb A^3_k$. 
If the characteristic of $k$ is positive, 
then the symbolic Rees ring $R_S(P)=\bigoplus_{m=0}^{\infty} P^{(m)}$ is finitely generated. 
\end{conje}

It is known that Conjecture \ref{GW} is reduced to the case where $k=\overline{\F}_p$. 
On the other hand, Theorem \ref{intro-t-main} indicates that 
a symbolic Rees algebra is not necessarily finitely generated in a higher dimensional case, 
even if the base field is $\overline{\F}_p$. 
Thus if the Conjecture \ref{GW} holds true, 
then its proof depends on some facts that hold only in a lower dimensional situation.

\begin{ack}
The authors would like to thank Professors Kazuhiko Kurano and Shinnosuke Okawa 
for several useful comments and discussion. 
We are grateful to the referee for valuable comments. 
The first author would like to thank Professor Shigeru Mukai for his warm 
encouragement and stimulating discussions. The first author was partially supported
by JSPS Grant-in-Aid (S) No 25220701 and JSPS Grant-in-Aid for Young Scientists (B) 16K17581.
The second author was funded by EPSRC. 
\end{ack}

\section{Preliminaries}

\subsection{Notation} 

In this subsection, we summarise notation used in the paper.

We say that $X$ is a {\em variety} over a field $k$ (or a $k$-{\em variety}) 
if $X$ is an integral scheme which is separated and of finite type over $k$. 
We say that $X$ is a {\em curve} over $k$ or a $k$-{\em curve}  
(resp. a {\em surface} over $k$ or a $k$-{\em surface})   
if $X$ is a variety over $k$ with $\dim X=1$ (resp. $\dim X=2$).

Given an invertible sheaf $L$ on a proper scheme $X$ over a field $k$, 
consider the natural homomorphism: 
\begin{equation}\label{e-nota1}
H^0(X, L) \otimes_k \MO_X \to L.
\end{equation}
\begin{enumerate}
\item 
We say that $L$ is {\em nef} if $L \cdot C \geq 0$ for any $k$-curve $C$ on $X$. 
\item 
For a $k$-linear subspace $V$ of $H^0(X, L)$, 
the {\em scheme-theoretic base locus} $B(V)$ of $V$ is the closed subscheme of $X$ 
defined by the image of the composite homomorphism
$$V \otimes_k L^{-1} \hookrightarrow H^0(X, L) \otimes_k L^{-1} \to \MO_X,$$
where the latter one is induced by (\ref{e-nota1}). 
For the linear system $\Lambda$ corresponding to $V$, 
we set $B(\Lambda):=B(V)$. 
\item 
We say that $L$ is {\em globally generated} if 
(\ref{e-nota1}) is surjective, i.e. $B(|L|)=\emptyset$. 
\item 
We say that $L$ is {\em semi-ample} if 
there exists a positive integer $n$ such that $L^{\otimes n}$ is globally generated. 
\end{enumerate}
For a $\Q$-Cartier $\Q$-divisor $D$ on a normal proper variety $X$ over a field, 
we say that $D$ is {\em nef} (resp. {\em semi-ample}) if there exists a positive integer $n$ 
such that $nD$ is a Cartier divisor and $\MO_X(nD)$ is nef (resp. semi-ample).

\subsection{Cox rings}
In this subsection, 
we recall definition of Cox rings (Definition \ref{d-Cox}) and a basic property (Lemma \ref{l-MDS}).

\begin{dfn}\label{d-multi-sec}
Let $k$ be a field. 
Let $X$ be a normal variety over $k$. 
For a subsemigroup $\Gamma$ of the group ${\rm WDiv}\,(X)$ of Weil divisors,  
we set 
$$R(X, \Gamma):=\bigoplus_{D \in \Gamma} H^0(X, \MO_X(D)),$$
which is called {\em the multi-section ring} of $\Gamma$. 
\end{dfn}

\begin{dfn}\label{d-Cox}
Let $k$ be a field. 
Let $X$ be a proper normal variety over $k$ 
whose divisor class group ${\rm Cl}(X)$ is a finitely generated free abelian group. 
Fix a subgroup $\Gamma$ of the group ${\rm WDiv}\,(X)$ of Weil divisors such that 
the induced group homomorphism $\Gamma \to {\rm Cl}(X)$ is bijective. 
We set 
$${\rm Cox}\,(X):=R(X, \Gamma)=\bigoplus_{D \in \Gamma} H^0(X, \MO_X(D)),$$
which is called the {\em Cox ring} of $X$. 
\end{dfn}

\begin{rem}\label{r-Cox}
If we take another subgroup $\Gamma'$ satisfying the same property as $\Gamma$, 
then it is known that $R(X, \Gamma)$ and $R(X, \Gamma')$ 
are isomorphic as $k$-algebras (cf. \cite[Remark 2.17]{GOST15}). 
\end{rem}

\begin{lem}\label{l-MDS}
Let $k$ be a field. 
Let $X$ be a projective normal $\Q$-factorial variety over $k$ 
whose divisor class group ${\rm Cl}(X)$ is a finitely generated free abelian group. 
Assume that 
\begin{enumerate}
\item[(a)] $X$ is geometrically integral over $k$, 
\item[(b)] $X$ is geometrically normal over $k$, 
\item[(d)] ${\rm Cox}\,(X)$ is a noetherian ring, and 
\item[(c)] ${\bf Pic}^0_X$ has dimension zero, 
where ${\bf Pic}^0_X$ denotes the identity component of the Picard scheme of $X$ over $k$ 
(cf. \cite[Remark 2.4]{Oka16}). 
\end{enumerate}
Then, the following assertions hold. 
\begin{enumerate}
\item 
For any finitely generated subsemigroup $\Gamma_1$ of ${\rm WDiv}\,(X)$, 
the multi-section ring $R(X, \Gamma_1)$ of $\Gamma_1$ is a finitely generated $k$-algebra.
\item An arbitrary nef Cartier divisor $L$ on $X$ is semi-ample. 
\end{enumerate}
\end{lem}

\begin{proof}
By (a) and (b), $X$ is a variety in the sense of 
\cite[the end of Section 1]{Oka16}. 
Then the conditions (c) and (d) enable us to apply \cite[Theorem 2.19]{Oka16}, 
hence $X$ is a Mori dream space in the sense of 
\cite[Definition 2.3]{Oka16}. 
Then (2) follow from \cite[Definition 2.3(2)]{Oka16}. 
Let us prove (1). By standard arguments (cf. \cite[discussion in Remark 2.17]{GOST15}), 
we may assume that $\Gamma_1$ is a subgroup of $\Gamma$ 
for some subgroup $\Gamma$ of ${\rm WDiv}\,(X)$. 
Then the assertion (2) holds by \cite[Lemma 2.20]{Oka16}. 
\end{proof}

\subsection{Symbolic Rees algebras}\label{ss-SRA}

The purpose of this subsection is to prove Proposition \ref{p-double-gr}, 
which gives a relation between symbolic Rees algebras of polynomial rings and 
Cox rings of blowups of projective spaces. 
The materials treated in this subsection might be well-known for experts, 
however we give the details of the proofs for the sake of completeness.

\begin{nota}\label{n-sra}
\begin{enumerate}
\item[(i)] 
Let $k$ be a field and 
let $A:=k[x_0, \cdots, x_n]$ be the polynomial ring equipped with 
the standard structure of a graded ring. 
Let $M$ be the homogenous maximal ideal of $A$. 
We have $\mathbb P^n_k=\Proj\,A$. 
\item[(ii)] 
Let $W$ be an integral closed subscheme of $\mathbb P^n_k$ and 
let $f:X \to \mathbb P^n_k$ be the blowup along $W$. 
For $D:=f^*\MO_{\mathbb P^n_k}(1)$ and the exceptional Cartier divisor $E$ 
that is the inverse image of $W$, we set 
$$R(X; D, -E):=\bigoplus_{d, e \in \Z_{\geq 0}} H^0(X, dD-eE).$$
\item[(iii)] 
There exists a homogeneous prime ideal $\p$ of $A:=k[x_0, \cdots, x_n]$ 
that induces the ideal sheaf on $\mathbb P^n_k$ corresponding to $W$. 
The {\em symbolic Rees algebra} of $\p$ is defined as  
$\bigoplus_{d=0}^{\infty}\p^{(d)},$ 
where $\p^{(d)}:= \p^d A_{\p} \cap A.$ 
\item[(iv)] 
Let $\mathcal I_W$ be the ideal sheaf on $\mathbb P^n_k$ 
corresponding to $W$. 
\end{enumerate}
\end{nota}

\begin{dfn}\label{d-saturation}
We use Notation \ref{n-sra}. 
For a homogenous ideal $I$ of $A$, we define the {\em saturation} $I^{\sat}$ of $I$ by 
$$I^{\sat}:=\bigcup_{\nu=1}^{\infty}\left\{ x \in A \,\middle|\, M^{\nu}x \subset I \right\}.$$
\end{dfn}

\begin{rem}\label{r-saturation}
We use the same notation as in Definition \ref{d-saturation}. 
By \cite[Excercise 5.10 in Ch. II]{Har77}, 
$I^{\sat}$ is a homogeneous ideal of $A$ such that 
both $I$ and $I^{\sat}$ define the same closed subscheme on $\mathbb P^n_k$ and 
the equation 
$$I^{\sat} = \bigoplus_{d=0}^{\infty} H^0(\mathbb P^n_k, \mathcal I(d))$$ 
holds, where $\mathcal{I}$ is the ideal sheaf on $\mathbb P^n_k$ associated with $I$. 
\end{rem}

\begin{dfn}\label{d-RR}
Let $R$ be a noetherian ring and let $J$ be an ideal of $R$. 
We define $\widetilde J$, called {\em the Ratliff--Rush ideal associated with} $J$, by 
$$\widetilde J:=\bigcup_{n=0}^{\infty}(J^{n+1}:J^n).$$ 
The ideal $J$ is said to be {\em Rattlif--Rush} if $J=\widetilde J$. 
It is well-known that $\widetilde{J}$ is a Ratliff--Rush ideal (cf. \cite[Introduction]{HLS92}). 
\end{dfn}

\begin{lem}\label{l-primary}
We use Notation \ref{n-sra}. 
Fix a positive integer $e$ 
and let $\p^e= \bigcap_{i=0}^r \q_i$ be a minimal primary decomposition of $\p^e$ 
such that $\sqrt{\q_0}=\p$ (cf. \cite[Section 4]{AM69}). 
Then the following hold. 
\begin{enumerate}
\item The equation $\p^{(e)}=\q_0$ holds. 
\item  The equation $(\p^e)^{\sat}= \bigcap_{i \in L} \q_i$ holds, 
where 
$$L:=\{i \in \{0, \cdots, r\}\,|\, \sqrt{\q_i} \neq M\}.$$ 
\end{enumerate}
\end{lem}

\begin{proof}
We show (1). 
Since $\p$ is a minimal prime ideal of $\p^e$, 
it follows from \cite[Proposition 4.9]{AM69} that $\p^eA_{\p}=\q_0A_{\p}$. 
In particular we get equations:  
$$\p^{(e)}=\p^eA_{\p} \cap A=\q_0A_{\p} \cap A=\q_0,$$
where the last equation follows from the fact that $\q_0$ is a $\p$-primary ideal. 
Thus (1) holds.

We show (2). 
First, let us prove $(\p^e)^{\sat} \subset \bigcap_{i \in L} \q_i$. 
Take $x \in (\p^e)^{\sat}$ and $i \in L$. 
By definition of the saturation $(\p^e)^{\sat}$ (cf. Definition \ref{d-saturation}), 
there is $\nu \in \Z_{>0}$ such that $M^{\nu}x \subset \p^e \subset \q_i$. 
As $\sqrt{\q_i} \neq M$, there is $y \in M \setminus \sqrt{\q_i}$. 
Hence $y^{\nu}x \in \q_i$. 
Since $\q_i$ is a primary ideal, 
it holds that $x \in\q_i$. 
Thus the inclusion $(\p^e)^{\sat} \subset \bigcap_{i \in L} \q_i$ holds. 

Second we prove the remaining inclusion: $(\p^e)^{\sat} \supset \bigcap_{i \in L} \q_i$. 
If $L=\{0, \cdots, r\}$, then there is nothing to show. 
We may assume that $L \neq \{0, \cdots, r\}$. 
As the primary decomposition $\p^e= \bigcap_{i=0}^r \q_i$ is minimal, 
there exists a unique index $i_1 \in \{1, \cdots, r\}$ such that $\sqrt{\q_{i_1}}=M$ (cf. \cite[Lemma 4.3]{AM69}). 
In particular, $L=\{0, \cdots, r\} \setminus \{i_1\}$.  
Since $A$ is a noetherian ring, 
there exists a positive integer $\nu$ such that $M^{\nu} \subset \q_{i_1}$. 
It follows from definition of the saturation $(\p^e)^{\sat}$ (cf. Definition \ref{d-saturation})
that $\bigcap_{i \in L} \q_i=\bigcap_{i \in \{0, \cdots, r\}, i \neq i_1} \q_i\subset (\p^e)^{\sat}$. 
\end{proof}

\begin{lem}\label{l-quasi-regular}
Let $R$ be a noetherian ring and 
let $I$ be an ideal of $R$ generated by a regular sequence $a_1, \cdots, a_{\mu}$ of $R$. 
Then the following hold. 
\begin{enumerate}
\item 
An $(R/I)$-algebra homomorphism
$$(R/I)[X_1,\dots, X_{\mu}] \to \bigoplus_{m=0}^{\infty} I^m/I^{m+1}, \quad X_i \mapsto a_i \mod I^2$$
is an isomorphism, where $I^0:=R$. 
\item 
If $I$ is a prime ideal of $R$ other than $\{0\}$, then $I^e$ is a Ratliff--Rush ideal for any positive integer $e$
 (cf. Definition \ref{d-RR}). 
\item 
If $I$ is a prime ideal of $R$, then for any positive integer $e$, 
an arbitrary associated prime ideal of $I^e$ is equal to $I$. 
\end{enumerate}
\end{lem}

\begin{proof}
The assertion (1) holds by the fact that 
any regular sequence is quasi-regular (\cite[Theorem 16.2(i)]{Mat89}). 
The assertion (2) follows from (1) and \cite[(1.2)]{HLS92}. 

We show (3). 
By (1), $I^m/I^{m+1}$ is a free $(R/I)$-module for any $m \in \Z_{\geq 0}$. 
Consider an exact sequence:
\[
0 \to I^m/I^{m+1} \to R/I^{m+1} \to R/I^{m} \to 0.
\]
We deduce from induction on $e$ that for any $e \in \Z_{\geq 1}$, 
an arbitrary associated prime of $I^e$ is equal to $I$. 
Thus (3) holds. 
\end{proof}

\begin{lem}\label{l-ideal-power}
We use Notation \ref{n-sra}. 
Assume that $W$ is a local complete intersection scheme. 
Fix a positive integer $e$. 
Then the equation $f_*\mathcal{O}_X(-eE) = \mathcal{J}^e$ 
holds as subsheaves of $\MO_{\mathbb P^n_k}$. 
\end{lem}

\begin{proof}
Fix a point $z \in \mathbb P^n_k$ and set $R:=\MO_{\mathbb P^n_k, z}.$ 
Given a positive integer $e$, let 
\begin{eqnarray*}
I&:=&\Gamma(\Spec\,R, \mathcal J|_{\Spec\,R})\\
R(I^e)&:=&\bigoplus_{d=0}^{\infty} I^{ed}\\
g_e:Y_e=\Proj\,R(I^e) &\to& \Spec\,R,
\end{eqnarray*}
where $I^0:=R$ and $g_e$ is the blowup along $I^e$. 
We set $Y:=Y_1$ and $g:=g_1$. 
Let $E_e$ be the effective Cartier divisor such that $\MO_{Y_e}(-E_e):=I^e\MO_{Y_e}$. 
In particular, $E=E_1$. 
Thanks to \cite[Exercise 5.13 in Ch II]{Har77}, 
we have that $\rho_e:Y \xrightarrow{\simeq} Y_e$ and $(\rho_e)_*(eE)=E_e$. 
We get equations 
$$I^e=\widetilde{I^e}=H^0(Y_e, \MO_{Y_e}(-E_e))=H^0(Y, \MO_Y(-eE)),$$
where 
the first equation holds by Lemma \ref{l-quasi-regular}(2), 
the second one follows from \cite[Fact 2.1]{HLS92} and 
the third one is obtained by $\rho_e$. 
Hence we are done. 
\end{proof}

\begin{lem}\label{l-double-gr}
We use Notation \ref{n-sra}. 
Assume that $W$ is locally complete intersection. 
Then $R(X; D, -E)$ and $\bigoplus_{e=0}^{\infty}\p^{(e)}$ are isomorphic as $k$-algebras. 
\end{lem}

\begin{proof}
Fix a non-negative integer $e$. 
We show that $\bigoplus_{d=0}^{\infty}H^0(X, dD-eE)$ is isomorphic to $\p^{(e)}$. 
By Lemma \ref{l-ideal-power}, we have $ f_*\mathcal{O}_X(-eE) \simeq \mathcal{J}^e$. 
By the projection formula, 
we get 
$$ f_*\mathcal{O}_X(dD-eE) \simeq \mathcal{J}^e\otimes_{\MO_{\mathbb P^n_k}} \MO_{\mathbb P^n_k}(d) = \mathcal{J}^e(d).$$
Thanks to Remark \ref{r-saturation}, we obtain an isomorphism: 
$$(\p^e)^{\sat} \simeq \bigoplus_{d=0}^{\infty}H^0(X, dD-eE).$$ 

\begin{claim}\label{c-double-gr}
Any associated prime ideal of $\p^e$ is equal to either $\p$ or $M$. 
\end{claim}

\begin{proof}[Proof of Claim \ref{c-double-gr}]
Assume that there exists an associated prime ideal $\q$ of $\p^e$ 
other than $\p$ or $M$. 
Let us derive a contradiction. 
Since $\q \neq M=(x_0, \cdots, x_n)$, there is $x_{\ell}$ that is not contained in $\q$. 
Then $\q A_{x_{\ell}}$ is an associated prime ideal of $\p^e A_{x_{\ell}}$. 
Take a maximal ideal $\m$ of $A_{x_{\ell}}$ containing $\q A_{x_{\ell}}$. 
Then $\q A_{\m}$ is an associated prime ideal of $\p^e A_{\m}$ 
other than $\p A_{\m}$. 
Since $W$ is a local complete intersection scheme, 
we have that $\p A_{\m}$ is a prime ideal generated by a regular sequence, 
which contradicts Lemma \ref{l-quasi-regular}(3). 
This completes the proof of Claim \ref{c-double-gr}. 
\end{proof}
For a minimal primary decomposition $(\p^e)^{\sat}=\bigcap_{i=0}^{r} \q_i$ satisfying $\sqrt{\q_0} =\p$, 
we have that 
$$\p^{(e)}=\q_0=(\p^e)^{\sat} \simeq \bigoplus_{d=0}^{\infty}H^0(X, dD-eE),$$ 
where the first equation holds by Lemma \ref{l-primary}(1) and 
the second equation follows from Lemma \ref{l-primary}(2) and Claim \ref{c-double-gr}. 
This completes the proof of Lemma \ref{l-double-gr}
\end{proof}

\begin{prop}\label{p-double-gr}
We use Notation \ref{n-sra}. 
Assume that $W$ is smooth over $k$.  
Then the following are equivalent. 
\begin{enumerate}
\item $R(X; D, -E)$ is a noetherian ring. 
\item $\bigoplus_{e=0}^{\infty} \p^{(e)}$ is a noetherian ring. 
\item The Cox ring ${\rm Cox}(X)$ of $X$ is a noetherian ring. 
\end{enumerate}
\end{prop}

\begin{proof}
It follows from Lemma \ref{l-double-gr} that (1) is equivalent to (2). 
Since $X$ is the blowup of $\mathbb P^n_k$ along a smooth scheme $W$, 
the assumptions of Lemma \ref{l-MDS} hold. 
Then, thanks to Lemma \ref{l-MDS}(1), we have that (3) implies (1). 
Thus it suffices to show that (1) implies (3). 
Since it holds that $H^0(X, dD-eE)=0$ for $d \in \Z_{<0}$ and $e \in \Z$, 
we get an isomorphism: 
$$\bigoplus_{d, e \in \Z, d\geq 0} H^0(X, dD-eE) \xrightarrow{\simeq} \bigoplus_{d, e \in \Z} H^0(X, dD-eE).$$
Thus we have a natural inclusion: 
$$R(X; D, -E)=\bigoplus_{d, e \in \Z_{\geq 0}} H^0(X, dD-eE) \hookrightarrow \bigoplus_{d, e \in \Z, d\geq 0} H^0(X, dD-eE).$$
The right hand side is generated by $H^0(X, E)$ as an $R(X; D, -E)$-algebra. 
Therefore, if $R(X; D, -E)$ is a noetherian ring, then so is $\bigoplus_{d, e \in \Z} H^0(X, dD-eE)$. 
Hence, also ${\rm Cox}(X)$ is a noetherian ring. 
Thus (1) implies (3). 
\end{proof}

\section{The main theorem}\label{s-main-thm}

\subsection{Construction in a general setting}

The purpose of this subsection is to give a sufficient condition 
under which the blowup of a smooth subvariety in a projective space 
has a nef Cartier divisor that is not semi-ample (Notation \ref{n-1}, Proposition \ref{p-not-sa}).

\begin{nota}\label{n-1}
We use notation as follows. 
\begin{enumerate}
\item[(i)] 
Let $k$ be a field.  
We work over $k$ unless otherwise specified 
(e.g. a projective scheme means a scheme that is projective over $k$). 
\item[(ii)] 
Let $V$ be a smooth projective variety. 
Set $d:=\dim V$. 
\item[(iii)] 
Let $M$ be a nef Cartier divisor on $V$ which is not semi-ample. 
\item[(iv)] 
Fix a closed immersion: $V \subset \mathbb P^n_k$. 
Let $H$ be a very ample Cartier divisor such that $\MO_{\mathbb P^n_k}(H) \simeq \MO_{\mathbb P^n_k}(1)$. 
We set $H_V$ to be the pullback of $H$ to $V$. 
\item[(v)] 
Assume that there exists a positive integer $r$ satisfying the following property: 
if $\Lambda$ denotes the linear system of $H^0(\mathbb P^n_k, \MO_{\mathbb P^n_k}(r))$ 
consisting of the effective divisors containing $V$, 
then the following conditions hold. 
\begin{enumerate}
\item[(v-1)] 
The base locus of $|\Lambda|$ is set-theoretically equal to $V$, 
i.e. for any point $y \in \mathbb P^n_k \setminus V$, 
there exists a hypersurface $S_0$ of $\mathbb P^n_{k}$ of degree $r$ 
such that $V \subset S_0$ and $y \not\in S_0$. 
\item[(v-2)] 
For any closed point $y \in V$, 
there exist an open neigbourhood $U$ of $y \in \mathbb P^n_k$ 
and hypersurfaces $S_1, \cdots, S_{n-\dim V}$ of $\mathbb P^n_k$ of degree $r$ 
such that $V$ is contained in $S_1 \cap \cdots \cap S_{n-\dim V}$ 
and that two subschemes $V \cap U$ and $S_1 \cap \cdots \cap S_{n-\dim V} \cap U$ 
of $\mathbb P^n_k$ are coincide. 
\end{enumerate}
\item[(vi)] 
Assume that there are a smooth prime divisor $W$ on $V$ and 
positive integers $s$ and $t$ satisfying the following properties. 
\begin{enumerate}
\item[(vi-1)] 
$st>r$. 
\item[(vi-2)] 
$W \sim stH_V-tM.$ 
\end{enumerate}
\item[(vii)]  
Let $f:X \to \mathbb P^n_k$ be the blowup along $W$. 
We set $V':=f_*^{-1}V$, $E:=\Ex(f)$ and 
$$S':=rf^*H-E.$$
Note that $E$ is a smooth prime divisor on $X$. 
Let $g:V' \xrightarrow{\simeq} V$ be the induced isomorphism. 
\item[(viii)] 
Set 
$$L:=(st-r)f^*H+S'.$$
\end{enumerate}
\end{nota}

\begin{lem}\label{l-etale-local}
Let $k$ be a field and let $Y:=\mathbb A^n_k=\Spec\,k[y_1, \cdots, y_n]$ be the $n$-dimensional affine space. 
For $i \in \{1, \cdots, n\}$, set $T_i:=V(y_i)$ to be the coordinate hyperplane of $Y=\mathbb A^n_k$. 
Let $q$ be a positive integer satisfying $q\leq n-1$. 
Set $V:=T_1 \cap \cdots \cap T_q$ and $W:=T_1 \cap \cdots \cap T_{q+1}$. 
Let $f:X \to Y$ be the blowup along $W$ and 
let $V'$ and $T'_i$ be the proper transforms of $V$ and $T_i$, respectively. 
Then an equation $V'=T'_1 \cap \cdots \cap T'_q$ holds. 
\end{lem}

\begin{proof}
Since blowups are commutative with flat base changes, we may assume that $q=n-1$. 
Thus $W$ is the origin and $V$ is a line passing through $W$. 
The inclusion $V' \subset T'_1 \cap \cdots \cap T'_{n-1}$ is clear, 
hence it suffices to prove that $T'_1 \cap \cdots \cap T'_{n-1} \cap E$ is one point, 
where $E$ denotes the $f$-exceptional prime divisor. 
To prove this, we may assume that $k$ is algebraically closed. 
Then $T'_1 \cap \cdots \cap T'_{n-1} \cap E$ is one point, since there is a canonical bijection between the set $E(k)$ of 
the closed points of $E$ and the set of the lines on $\mathbb P^n_k$ passing through $W$. 
\end{proof}

\begin{prop}\label{p-not-sa}
We use Notation~\ref{n-1}. 
Then the following hold. 
\begin{enumerate}
\item 
The base locus of the complete linear system $|S'|$ 
is contained in $V'$. 
\item 
$L|_{V'} \sim t g^*M$. 
\item 
$L$ is a nef Cartier divisor which is not semi-ample. 
\end{enumerate}
\end{prop}

\begin{proof}
We show (1). 
Take a closed point $x \in X \setminus V'$. 
We set $y:=f(x)$. 
It suffices to show that the base locus $B(|S'|)$ of $|S'|$ does not contain $x$. 
We separately treat the following two cases: $y \not\in V$ and $y \in V$. 

Assume that $y \not\in V$. 
By Notation~\ref{n-1}(v-1), 
there exists a hypersurface $S_0$ of $\mathbb P^n_k$ of degree $r$ 
such that $V \subset S_0$ and $y \not\in S_0$. 
It holds that 
$$rf^*H \sim f^*S_0=S'_0+aE,$$
where $a \in \Z_{>0}$ and $S'_0$ is the proper transform of $S_0$. 
In particular, we have that 
$$B(|S'|) \subset \Supp(S'_0+E)=f^{-1}(S_0).$$ 
It follows from $y \not\in S_0$ that $x \not\in f^{-1}(S_0)$. 
Hence, $x \not\in B(|S'|)$. 
This completes the proof for the case where $y \not\in V$. 

Assume that $y \in V$. 
We have that $x \in E \setminus V'$. 
By Notation~\ref{n-1}(v-2), 
there exist an open neigbourhood $U$ of $y \in \mathbb P^n_k$ 
and hypersurfaces $S_1, \cdots, S_{n-\dim V}$ of $\mathbb P^n_k$ of degree $r$ 
such that $V$ is contained in $S_1 \cap \cdots \cap S_{n-\dim V}$ 
and that two subschemes $V \cap U$ and $S_1 \cap \cdots \cap S_{n-\dim V} \cap U$ 
of $\mathbb P^n_{k}$ are the same. 
In particular, $S_1, \cdots, S_{n-\dim V}$ are smooth at $y$ and 
form a part of a regular system of parameters of $\MO_{\mathbb P^n_k, y}$ (cf. \cite[Theorem 17.4]{Mat89}). 
Therefore, thanks to Cohen's structure theorem, the situation is the same, 
up to taking the formal completions, as in the statement of Lemma \ref{l-etale-local}. 
It follows from Lemma \ref{l-etale-local} and the faithfully flatness of completions 
(cf. \cite[Theorem 7.5(ii)]{Mat89}) that an equation 
$$V' \cap f^{-1}(U)=S'_1 \cap \cdots \cap S'_{n-\dim V}\cap f^{-1}(U)$$
holds, where each $S'_i$ denotes the proper transform of $S_i$. 
In particular, it holds that $x \notin S'_{i_0}$ for some $i_0 \in \{1, \cdots, n-\dim V\}$. 
Since $S'_{i_0}$ is smooth at a point $y$ of $W$, we have that 
$$S'=f^*(rH)-E \sim f^*S_{i_0}-E=S'_{i_0}.$$
Thus, in any case, the base locus $B(|S'|)$ does not contain $x$. 
Hence, (1) holds. 

The assertion (2) holds by the following computation:  
\begin{eqnarray*}
L|_{V'}&=&\left((st-r)f^*H+S'\right)\Big{|}_{V'}\\
&\sim& g^*\left((st-r)H_V+(S|_V-W)\right) \\
&\sim& g^*\left((st-r)H_V+(rH_V-(stH_V-tM))\right) \\
&\sim& t g^*M.
\end{eqnarray*}

We show (3). 
Since $L|_{V'}$ is not semi-ample by (2) and Notation~\ref{n-1}(iii), neither is $L$. 
Thus it suffices to show that $L=(st-r)f^*H+S'$ is nef. 
Take a curve $\Gamma$ on $X$. 
If $\Gamma \not\subset V'$, 
then we get $((st-r)f^*H+S') \cdot \Gamma \geq 0$ by (1). 
If $\Gamma \subset V'$, then (2) implies that $L \cdot \Gamma \geq 0$. 
In any case, we obtain $L \cdot \Gamma \geq 0$, 
and hence $L$ is nef. 
Thus (3) holds. 
\end{proof}

\subsection{Proof of the main theorem}

In this subsection, we prove the main theorem of this paper (Theorem \ref{t-main2}). 
Theorem \ref{t-main2} is a formal consequence of Theorem \ref{t-main} and some results established before. 
The main part of Theorem \ref{t-main} is to find schemes and divisors satisfying Notation~\ref{n-1}. 
To this end, we start with the following lemma. 

\begin{lem}\label{l-ample-connected}
Let $k$ be a field. 
Let $V$ be a smooth projective connected scheme over $k$ such that $\dim V\geq 2$. 
Let $W$ be an ample effective Cartier divisor. 
Then $W$ is connected. 
\end{lem}

\begin{proof}
Set $k':=H^0(V, \MO_V)$. 
Note that $k \subset k'$ is a field extension of finite degree. 
We have natural morphisms: 
\[
\alpha:V \xrightarrow{\alpha'} \Spec\,k' \xrightarrow{\beta} \Spec\,k.
\]
We obtain $\alpha'_*\MO_V=\MO_{\Spec\,k'}$. 

Let us prove that $k \subset k'$ is a separable extension. 
It suffices to prove that $A:=k' \otimes_k \overline k$ is reduced for an algebraic closure $\overline k$ of $k$. 
We have the induced morphism 
\[
\alpha''=\alpha' \times_k \overline k:V \times_k \overline k \to \Spec\,(k' \otimes_k \overline k)=\Spec\,A.
\]
Since $k \to \overline k$ is flat, we have that $\alpha''_*\MO_{V \times_k \overline k}=\MO_{\Spec\,A}$. 
As $V \times_k \overline k$ is reduced, so is $A$. 
Therefore, $k \subset k'$ is a separable extension. 

We have that $\alpha$ is smooth and $\beta$ is \'etale. 
Then it holds that also $\alpha'$ is smooth by \cite[Proposition 2.4.1]{Fu15}. 
Therefore, the problem is reduced to the case where $k=H^0(V, \MO_V)$.

We are allowed to replace $W$ by $nW$ for a positive integer $n$. 
Hence, by Serre duality and the ampleness of $W$, we may assume that 
$H^1(V, \MO_V(-W))=0$. 
Then we obtain a surjective $k$-linear map 
\[
H^0(V, \MO_V) \to H^0(W, \MO_W). 
\]
Since $\dim_k H^0(V, \MO_V)=1$, we get $\dim_k H^0(W, \MO_W)=1$. 
Therefore, $W$ is connected. 
\end{proof}

\begin{lem}\label{l-totaro}
The following hold. 
\begin{enumerate}
\item 
Let $n$ be an integer such that $n \geq 5$. 
If $k$ is an algebraically closed field, 
then there exist a smooth projective surface $V$ over $k$, 
a closed immersion $j:V \hookrightarrow \mathbb P^{n}_k$ over $k$, 
and a nef Cartier divisor $M$ on $V$ which is not semi-ample. 
\item 
Let $n$ be an integer such that $n \geq 11$. 
If $k$ is a field, 
then there exist a smooth projective surface $V$ over $k$, 
a closed immersion $j:V \hookrightarrow \mathbb P^n_k$ over $k$, 
and a nef Cartier divisor $M$ on $V$ which is not semi-ample. 
\end{enumerate}
\end{lem}

\begin{proof}
We show (1). 
We may assume that $n=5$. 
The existence of $j$ is automatic, since any smooth projective surface over $k$ 
can be embedded in $\mathbb P^5_k$. 
If $k$ is the algebraic closure of a finite field, 
then the assertion follows from \cite[Theorem 6.1]{Tot09}. 
If $k$ is not algebraic over any finite field, then $V$ can be taken as the direct product of an elliptic curve $E$ and a smooth projective curve. 
Indeed, there is a Cartier divisor $N$ on $E$ such that $\deg N=0$ and $N$ is not torsion, i.e. $rN \not\sim 0$ for any positive integer $r$. 
This implies that $N$ is a nef Cartier divisor which is not semi-ample. 
Hence, its pullback $M$ to $V$ is again a nef Cartier divisor which is not semi-ample.  
This completes the proof of (1). 

We show (2). We may assume that $n=11$. 
First we treat the case where $k$ is a perfect field. 
By (1), we can find a field extension $k \subset k'$ of finite degree, 
a connected $k'$-scheme $V$ of dimension two which is smooth and projective over $k'$, 
a closed immersion $j':V \hookrightarrow \mathbb P^{5}_{k'}$ over $k'$ 
and a nef Cartier divisor $M$ on $V$ which is not semi-ample. 
Automatically $V$ is projective over $k$. 
Since $k$ is perfect, $V$ is also smooth over $k$. 
Thus it suffices to find a closed immersion $j:V \hookrightarrow \mathbb P^{11}_{k}$ over $k$. 
Since $k \subset k'$ is a finite separable extension, it is a simple extension. 
Therefore, there is a closed immersion $i:\Spec\,k' \hookrightarrow \mathbb P^1_k$ over $k$.
We can find a required closed immersion $j$ by using the Segre embedding: 
$$j:V \overset{j'}\hookrightarrow \mathbb P^{5}_{k'} 
= \mathbb P^{5}_{k} \times_k k' \overset{{\rm id} \times i}\hookrightarrow 
\mathbb P^{5}_{k} \times_k \mathbb P^1_k \overset{\text{Segre}}\hookrightarrow \mathbb P^{11}_{k}.$$
This completes the proof of the case where $k$ is a perfect field. 

Second we handle the general case. 
Let $k_0$ be the prime field contained in $k$. 
Since $k_0$ is perfect, 
there exist a smooth projective connected $k_0$-scheme $V_0$ of dimension two, 
a closed immersion $j_0:V_0 \hookrightarrow \mathbb P_{k_0}^{11}$ over $k_0$, 
and a nef Cartier divisor $M_0$ on $V_0$ which is not semi-ample. 
Then $V_0 \times_{k_0} k$ is a scheme which is smooth and projective over $k$. 
Since any ring homomorphism between fields is faithfully flat, 
we can find a connected component $V$ of $V_0 \times_{k_0} k$ such that $M:=(\alpha^*M_0)|_V$ is not semi-ample, where $\alpha:V_0 \times_{k_0} k \to V_0$. 
Since $M_0$ is nef, so is $M$ (cf. \cite[Lemma 2.3]{Tan}).  
Clearly, $V$ is a smooth projective surface over $k$ and 
there is a closed immersion $j:V \hookrightarrow \mathbb P_{k}^{11}$ over $k$. 
This completes the proof of (2). 
\end{proof}

\begin{thm}\label{t-main}
The following hold. 
\begin{enumerate}
\item 
Let $n$ be an integer such that $n \geq 5$. 
If $k$ is an algebraically closed field, 
then there exist a one-dimensional connected closed subscheme $W$ of $\mathbb P^n_k$ 
which is smooth over $k$ and a Cartier divisor $L$ on the blowup $X$ of $\mathbb P^n_k$ along $W$ 
such that $L$ is nef but not semi-ample. 
\item 
Let $n$ be an integer such that $n \geq 11$. 
If $k$ is a field, 
then there exist a one-dimensional connected closed subscheme $W$ of $\mathbb P^n_k$ 
which is smooth over $k$ and a Cartier divisor $L$ on the blowup $X$ of $\mathbb P^n_k$ along $W$ 
such that $L$ is nef but not semi-ample. 
\end{enumerate}
\end{thm}

\begin{proof}
We only show (2), as the proof of (1) is easier. 
Fix a field $k$. 
We will find schemes and divisors satisfying the properties of Notation~\ref{n-1}. 
Thanks to Lemma~\ref{l-totaro}, 
there exist a smooth projective connected $k$-scheme $V$ of dimension two, 
a closed immersion $j:V \hookrightarrow \mathbb P^n_k$ over $k$, 
and a nef Cartier divisor $M$ on $V$ which is not semi-ample. 
Set $d:=2$. 
Then $k, V, M, d, n$ satisfy properties (i)-(iv) of Notation~\ref{n-1}. 

Since $V=\Proj\,k[x_0, \cdots, x_n]/(h_1, \cdots, h_a)$, 
it holds that the linear system $\Lambda$ appearing in Notation~\ref{n-1}(v) 
satisfies the property (v-1) of Notation~\ref{n-1} if $r \geq \max_{1 \leq q \leq a} \deg h_q$. 
As $V$ is a locally completion intersection scheme, 
the quasi-compactness of $V$ implies that 
also the property (v-2) of Notation~\ref{n-1} holds for $r \gg 0$. 
Therefore, we can find $r\in \Z_{>0}$ satisfying the property in (v) of Notation~\ref{n-1}.

We now show that there exist $s, t, W$ satisfying the property (vi) of Notation~\ref{n-1}. 
If $k$ is an infinite field, then the Bertini theorem 
enables us to find a positive integer $s$ and a smooth effective divisor $W$ on $V$ such that 
$W \sim sH_V-M$. 
Note that $W$ is connected (Lemma \ref{l-ample-connected}). 
Thus, $s$, $t:=1$ and $W$ satisfy the the property (vi) of Notation~\ref{n-1}. 
If $k$ is a finite field, then it follows from \cite[Theorem 1.1]{Poo04} 
that there are positive integers $t \gg s \gg 0$ and a smooth effective divisor $W$ 
satisfying the the property (vi) of Notation~\ref{n-1}. 
Again by Lemma \ref{l-ample-connected}, $W$ is connected. 
In any case, we can find $s, t, W$ satisfying the property (vi) of Notation~\ref{n-1}. 

To summarise, we have found $V, W, M, d, n, r, s, t$ over a field $k$ 
satisfying the properties (i)-(viii) of Notation~\ref{n-1}. 
By construction, $V$ is a smooth projective surface. 
In particular, $W$ is a smooth projective curve in $\mathbb P^{11}_k$. 
Thanks to Proposition~\ref{p-not-sa}, 
the Cartier divisor 
$$L=(st-r)f^*H+S'$$
on $X$, defined in (viii) of Notation~\ref{n-1}, is nef but not semi-ample.  
\end{proof}

\begin{thm}\label{t-main2}
The following hold. 
\begin{enumerate}
\item 
Let $q$ be an integer such that $q \geq 6$. 
If $k$ is an algebraically closed field, 
then there exists a homogeneous prime ideal $\p$ of the polynomial ring $k[x_1, \cdots, x_q]$ 
with $q$ variables whose symbolic Rees algebra $\bigoplus_{m=0}^{\infty} \p^{(m)}$ is not a noetherian ring. 
\item 
Let $q$ be an integer such that $q \geq 12$. 
If $k$ is a field, 
then there exists a homogeneous prime ideal $\p$ of the polynomial ring $k[x_1, \cdots, x_q]$ 
with $q$ variables whose symbolic Rees algebra $\bigoplus_{m=0}^{\infty} \p^{(m)}$ is not a noetherian ring. 
\end{enumerate}
\end{thm}

\begin{proof}
The assertion follows from Lemma \ref{l-MDS}, Proposition~\ref{p-double-gr} and Theorem~\ref{t-main}. 
\end{proof}

\end{document}